\documentclass[11pt,pdftex, reqno]{amsart}

\usepackage{amsmath, amsthm}
\usepackage{eucal}

\usepackage{palatino}
\usepackage{euler}

\usepackage{amssymb}
\usepackage{amscd}
\usepackage{latexsym}
\usepackage{epsfig}
\usepackage{graphicx, xcolor}
\usepackage{amsfonts}
\usepackage{psfrag}
\usepackage{caption}
\usepackage{setspace}
\usepackage{fullpage}
\usepackage{draftcopy}
\usepackage{tikz}

\usepackage{pifont}

\usepackage{verbatim}

\usepackage{dsfont}

\usepackage{url}
\usepackage{cite}

\usepackage{dsfont}

\usepackage[margin=1in]{geometry}

\input xy
\xyoption{all}
\UseComputerModernTips

\numberwithin{equation}{section}
\numberwithin{figure}{section}

\newtheorem{theorem}{Theorem}[section]

\newtheorem{lemma}[theorem]{Lemma}

\theoremstyle{definition}

\newcommand{\C}{{\mathbb{C}}}

\newcommand{\fl}{\mathscr{F}\ell}

\newcommand{\defi}{\textbf}
\newcommand{\dom}{\backslash}

\newcommand{\init}{\mathop{\mathrm{init}}}

\title{Generating the Ideals Defining Unions of Schubert Varieties}
\author{Anna Bertiger}
\date{\today}                            
\begin{document}

\maketitle

\begin{abstract}
This note computes a Gr\"obner basis for the ideal defining a union of Schubert varieties.  More precisely, it computes a Gr\"obner basis for unions of schemes given by northwest rank conditions on the space of all matrices of a fixed size.  Schemes given by northwest rank conditions include classical determinantal varieties and matrix Schubert varieties--closures of Schubert varieties lifted from the flag manifold to the space of matrices.
\end{abstract}

\section{Introduction}

We compute a Gr\"obner basis, and hence ideal generating set, for the ideal defining a union of schemes each given by northwest rank conditions with respect to an ``antidiagonal term order."  A \defi{scheme defined by northwest rank conditions} is any scheme whose defining equations are of the form ``all $k \times k$ minors in the northwest $i \times j$ sub-matrix of a matrix of variables," where $i,j,$ and $k$ can take varying values.  These schemes represent a generalization of classical determinantal varieties--those varieties with defining equations all $(r+1) \times (r+1)$ minors of a matrix of variables.  One geometrically important collection of schemes defined by northwest rank conditions is the set of matrix Schubert varieties.  Matrix Schubert varieties are closures of the lift of Schubert varieties from the complete flag manifold to matrix space \cite{Fulton1992}.  In general, a \defi{matrix Schubert variety} for a partial permutation $\pi$ is the subvariety of matrix space given by the rank conditions that the northwest $i \times j$ sub-matrix must have rank at most the number of $1$s in the northwest $i \times j$ sub-matrix of the partial permutation matrix for $\pi$.  Notice that the set of matrix Schubert varieties contains the set of classical determinantal varieties, which are the zero locus of all minors of a fixed size on the space of all matrices of fixed size.  

Matrix Schubert varieties associated to honest, that is non-partial, permutations are the closures of the lifts of the corresponding Schubert varieties in the flag manifold, $B_-\setminus GL_n$.  If $\overline{X}_\pi$ is the matrix Schubert variety for an honest permutation $\pi$ the projection 
\[
\{\text{full rank matrices}\} \twoheadrightarrow B_-\setminus GL_n\C=\fl \C^n
\]
sends $\overline{X}_\pi \cap GL_n\C$ onto the Schubert variety $X_\pi \subseteq \fl\C^n$.  Schubert varieties, orbits of $B_+$, stratify $\fl \C^n$ and give a basis for $H^*(\fl \C^n)$. It is this application that led to the introduction of matrix Schubert varieties in \cite{Fulton1992}.  Knutson and Miller showed that matrix Schubert varieties have a rich algebro-geometric structure corresponding to beautiful combinatorics \cite{KnutsonMiller 2005}.  Fulton's generators are a Gr\"obner basis with respect to any antidiagonal term order and their initial ideal is the Stanley-Reisner ideal of the ``pipe dream complex."   Further, Knutson and Miller show that  the pipe dream complex is shellable, hence the original ideal is Cohen-Macaulay.  Pipe dreams, the elements of the pipe dream complex, were originally called RC graphs and were developed by Bergeron and Billey \cite{BergeronBilley} to describe the monomials in polynomial representatives for the classes corresponding to Schubert varieties in $H^*(\fl \C^n)$. 

The importance of Schubert varieties, and hence matrix Schubert varieties, to other areas of geometry has become increasing evident.  For example, Zelevinsky \cite{Zelevinski} showed that certain quiver varieties, sequences of vector space maps with fixed rank conditions, are isomorphic to Schubert varieties.   Knutson, Miller and Shimozono, \cite{KMS} produce combinatorial formulae for quiver varieties using many combinatorial tools reminiscent of those for Schubert varieties.  

\subsection{Notation and Background}

Much of the background surveyed here can be found in \cite{MillerSturmfels}.  Let $B_-$ (respectively $B_+$) denote the group of invertible lower triangular (respectively upper triangular) $n \times n$ matrices.  Let $M=(m_{i,j})$ be a matrix of variables.  In what follows $\pi$ will be a possibly partial permutation, written in one-line notation $\pi(1)\ldots \pi(n),$  with entries for $\pi(i)$ undefined are written $\star$.   We shall write permutation even when we mean partial permutation in cases where there is no confusion.  A \defi{matrix Schubert variety} $\overline{X}_\pi$ is the closure $\overline{B_-\pi B_+}$ in the affine space of all matrices, where $\pi$ is a permutation matrix and $B_-$ and $B_+$ act by downward row and rightward column operations respectively.  Notice that for $\pi$ an honest permutation $\overline{X}_\pi$ is the closure of the lift of $X_\pi=B_- \dom \overline{B_- \pi B_+} \subseteq B_- \dom GL_n\C$ to the space of $n \times n$ matrices.  

The \defi{Rothe diagram} of a permutation is found by looking at the permutation matrix and crossing out all of the cells weakly below, and the cells weakly to the right of, each cell containing a $1$.  The remaining empty boxes form the Rothe diagram.  The \defi{essential boxes} \cite{Fulton1992} of a permutation are those boxes in the Rothe diagram that do not have any boxes of the diagram immediately south or east of them.  The Rothe diagrams for $2143$ and $15432$ are given in Figure \ref{fig:ess2143}.  In both cases the essential boxes are marked with the letter $e$.  

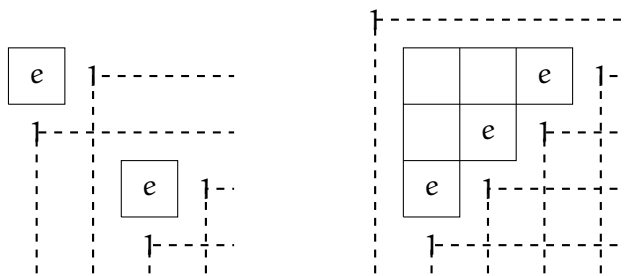
\begin{figure}[htbp]
\begin{center}
\begin{tikzpicture}[x=.75cm, y=.75cm] 
\draw (.5,2.5) node{$1$};
\draw [dashed, thick] (.5,0)--(.5,2.5)--(4,2.5);
\draw (1.5,3.5) node{$1$};
\draw [dashed, thick] (1.5,0)--(1.5,3.5)--(4,3.5);
\draw (3.5,1.5) node{$1$};
\draw [dashed, thick] (3.5,0)--(3.5,1.5)--(4,1.5);
\draw (2.5,.5) node{$1$};
\draw [dashed, thick] (2.5,0)--(2.5,.5)--(4,.5);
\draw (0,3)--(1,3)--(1,4)--(0,4)--(0,3);
\draw (.5,3.5) node{$e$};
\draw (2,1)--(3,1)--(3,2)--(2,2)--(2,1);
\draw (2.5,1.5) node{$e$};
\draw (6.5,4.5) node{$1$};
\draw [dashed, thick] (6.5,0)--(6.5,4.5)--(11,4.5);
\draw (10.5,3.5) node{$1$};
\draw [dashed, thick] (10.5,0)--(10.5,3.5)--(11,3.5);
\draw (9.5,2.5) node{$1$};
\draw [dashed, thick] (9.5,0)--(9.5,2.5)--(11,2.5);
\draw (8.5,1.5) node{$1$};
\draw [dashed, thick] (8.5,0)--(8.5,1.5)--(11,1.5);
\draw (7.5,.5) node{$1$};
\draw [dashed, thick] (7.5,0)--(7.5,.5)--(11,.5);
\draw (7,1)--(7,4)--(10,4)--(10,3)--(7,3);
\draw (7,2)--(9,2)--(9,4);
\draw (7,1)--(8,1)--(8,4);
\draw (7.5,1.5) node{$e$};
\draw (8.5,2.5) node{$e$};
\draw (9.5,3.5) node{$e$};
\end{tikzpicture}
\caption{The Rothe diagrams and essential sets of $2143$ (left) and $15432$ (right). }
\label{fig:ess2143}
\end{center}
\end{figure}

The \defi{rank matrix} of a permutation $\pi$, denoted $r(\pi)$, gives in each cell $r(\pi)_{ij}$ the rank of the $i \times j$ northwest-justified sub-matrix of the permutation matrix for $\pi$.  For example, the rank matrix of $15432$ is \[
\left(\begin{array}{ccccc}1 & 1 & 1 & 1 & 1 \\1 & 1 & 1 & 1 & 2 \\1 & 1 & 1 & 2 & 3 \\1 & 1 & 2 & 3 & 4 \\1 & 2 & 3 & 4 & 5\end{array}\right).
\]

\begin{theorem}[\cite{Fulton1992}]
Matrix Schubert varieties have radical ideal $I(\overline{X}_\pi)=I_\pi$ given by determinants representing conditions given in the rank matrix $r(\pi)$, that is, the $(r(\pi)_{ij}+1) \times (r(\pi)_{ij}+1)$ determinants of the northwest $i \times j$ sub-matrix of a matrix of variables.  In fact, it is sufficient to impose only those rank conditions $r(\pi)_{ij}$ such that $(i,j)$ is an essential box for $\pi$.  
\end{theorem}

Hereafter we call the determinants corresponding the to essential rank conditions, or the analogous determinants for any ideal generated by northwest rank conditions, the \defi{Fulton generators}.  

One special form of ideal generating set is a Gr\"obner basis.  To define a Gr\"obner basis we set a total ordering on the monomials in a polynomial ring such that $1 \le m$ and  $m <n$ implies $mp<np$ for all monomials $m$, $n$ and $p$.  Let $\init f$ denote the largest monomial that appears in the polynomial $f$.  A \defi{Gr\"obner basis} for the ideal $I$ is a set $\{f_1, \ldots f_r\} \subseteq I$ such that $\init I := \langle \init f : f \in I \rangle = \langle \init f_1, \ldots \init f_r \rangle.$  Notice that a Gr\"obner basis for $I$ is necessarily a generating set for $I$.  

The \defi{antidiagonal} of a matrix is the diagonal series of cells in the matrix running from the most northeast to the most southwest cell.  The \defi{antidiagonal term} (or \defi{antidiagonal}) of a determinant is the product of the entries in the antidiagonal.  For example, the antidiagonal of $\left(\begin{smallmatrix} a&b\\ c&d \end{smallmatrix}\right)$ is the cells occupied by $b$ and $c$, and correspondingly, in the determinant $ad-bc$ the antidiagonal term is $bc$.  Term orders that select antidiagonal terms from a determinant, called \defi{antidiagonal term orders} have proven especially useful in understanding ideals of matrix Schubert varieties.   There are several possible implementations of an antidiagonal term order on an $n \times n$ matrix of variables, any of which would suit the purposes of this paper.  One example is weighting the top right entry highest and decreasing along the top row before starting deceasing again at the right of the next row; monomials are then ordered by their total weight.   

\begin{theorem}[\cite{KnutsonMiller2005}]
The Fulton generators for $I_\pi$ form a Gr\"obner basis under any antidiagonal term order.  
\end{theorem}

Typically we will denote the cells of a matrix that form antidiagonals by $A$ or $B$.  In what follows if $A$ is the antidiagonal of a sub-matrix of $M$ we will use the notation $\det(A)$ to denote the determinant of this sub-matrix.  We shall be fairly liberal in exchanging antidiagonal cells and the corresponding antidiagonal terms, thus, for any antidiagonal term order, $A=\init \det(A)$.  

\subsection{Statement of Result}

Let $I_1, \ldots I_r$ be ideals defined by northwest rank conditions.  We will produce a Gr\"obner basis, and hence ideal generating set, for $I_1 \cap \cdots \cap I_r$.   For each list of antidiagonals $A_1, \ldots, A_r$, where $A_i$ is the antidiagonal of a Fulton generator of $I_{i}$, we will produce a Gr\"obner basis element $g_{A_1, \ldots , A_r}$ for $\cap I_i$.  The generators $g_{A_1, \ldots , A_r}$ will be products of determinants, though not simply the product of the $r$ determinants corresponding to the   $A_i$.  For a fixed list of antidiagonals $A_1, \ldots , A_r$, build the generator $g_{A_1, \ldots , A_r}$ by:

\begin{enumerate}
\item{Begin with $g_{A_1, \ldots , A_r}=1$}
\item{Draw a diagram with a dot of color $i$ in each box of $A_i$ and connect the consecutive dots of color $i$ with a line segment of color $i$.  }
\item{Break the diagram into connected components.  Two dots are connected if they are either connected by lines or are connected by lines to dots that occupy the same box.  }
\item{For each connected component, remove the longest series of boxes $B$ such that there is exactly one box in each row and column and the boxes are  all in the same connected component.  If there is a tie use the most northwest of the longest series of boxes.  Note that B need not be any of $A_1, \ldots , A_r$.  Multiply $g_{A_1, \ldots , A_r}$ by $\det(B)$.  Remove this antidiagonal from the diagram of the connected component, break the remaining diagram into components and repeat.  }
\end{enumerate}

\begin{theorem}\label{mainthm}
$\{g_{A_1\ldots A_r}: A_i \text{ is an antidiagonal of a Fulton generator of }I_i \text{, } 1 \le i \le r \}$ form a Gr\"obner basis, and hence a generating set, for $\cap_{i=1}^rI_i$.  
\end{theorem}

\subsection{Acknowledgements} This work constitutes a portion of my PhD thesis completed at Cornell University under the direction of Allen Knutson.   I wish to thank Allen for his help, advice and encouragement in completing this project.  Thanks also go to Jenna Rajchgot for helpful discussions in the early stages of this work.   I'd also like to thank the authors of computer algebra system Macaulay2, \cite{M2} which powered the computational experiments nessecary to do this work.  I'm especially grateful to Mike Stillman who patiently answered many of my Macaulay2 questions over the course of this work.  Kevin Purbhoo gave very helpful comments on drafts of this manuscript for which I cannot thank him enough.  

\section{Examples}

We delay the proof of Theorem \ref{mainthm} to Section \ref{S:proof} and first give some examples of the generators produced for given sets of antidiagonals.  These examples are given by pictures of the antidiagonals on the left and corresponding determinantal equations on the right.  Note that we only give particular generators, rather than entire generating sets, which might be quite large.  We then give entire ideal generating sets for two smaller intersections.  

If $r=1$ then for each Fulton generator with antidiagonal $A$ the algorithm produces the generator $g_A=\det(A)$.  Therefore, if we intersect only one ideal the algorithm returns the original set of Fulton generators.  The generator for the antidiagonal shown is exactly the determinant of the one antidiagonal pictured:
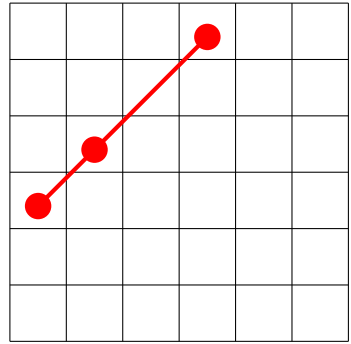
\begin{center}
\begin{tikzpicture}[x=.75cm, y=.75cm]
\draw (0,0)--(6,0);
\draw (0,1)--(6,1);
\draw (0,2)--(6,2);
\draw (0,3)--(6,3);
\draw (0,4)--(6,4);
\draw (0,5)--(6,5);
\draw (0,6)--(6,6);
\draw (0,0)--(0,6);
\draw (1,0)--(1,6);
\draw (2,0)--(2,6);
\draw (3,0)--(3,6);
\draw (4,0)--(4,6);
\draw (5,0)--(5,6);
\draw (6,0)--(6,6);
\draw [color=red, ultra thick] (.5,2.4)--(1.5,3.4)--(3.5,5.4);
\fill [color=red] (.5,2.4) circle(5pt);
\fill [color=red] (1.5,3.4) circle(5pt);
\fill [color=red] (3.5,5.4) circle(5pt);
\node [right] at (7,3) {$\left|\begin{array}{ccc}m_{1,1} & m_{1,2} & m_{1,4} \\m_{3,1} & m_{3,2} & m_{3,4} \\m_{4,1} & m_{4,2} & m_{4,4}\end{array}\right|.$};
\end{tikzpicture}
\end{center}

The generator for two disjoint antidiagonals is the product of the determinants corresponding to the two disjoint antidiagonals:
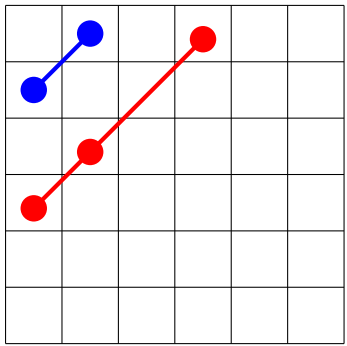
\begin{center}
\begin{tikzpicture}[x=.75cm, y=.75cm]
\draw (0,0)--(6,0);
\draw (0,1)--(6,1);
\draw (0,2)--(6,2);
\draw (0,3)--(6,3);
\draw (0,4)--(6,4);
\draw (0,5)--(6,5);
\draw (0,6)--(6,6);
\draw (0,0)--(0,6);
\draw (1,0)--(1,6);
\draw (2,0)--(2,6);
\draw (3,0)--(3,6);
\draw (4,0)--(4,6);
\draw (5,0)--(5,6);
\draw (6,0)--(6,6);
\draw [color=blue, ultra thick] (.5,4.5)--(1.5,5.5);
\fill [color=blue] (.5,4.5) circle(5pt);
\fill [color=blue] (1.5,5.5) circle(5pt);
\draw [color=red, ultra thick] (.5,2.4)--(1.5,3.4)--(3.5,5.4);
\fill [color=red] (.5,2.4) circle(5pt);
\fill [color=red] (1.5,3.4) circle(5pt);
\fill [color=red] (3.5,5.4) circle(5pt);
\node [right] at (7,3) {$\left|\begin{array}{cc}m_{1,1} & m_{1,2} \\m_{2,1} & m_{3,2}\end{array}\right|\left|\begin{array}{ccc}m_{1,1} & m_{1,2} & m_{1,4} \\m_{3,1} & m_{3,2} & m_{3,4} \\m_{4,1} & m_{4,2} & m_{4,4}\end{array}\right|.$};
\end{tikzpicture}
\end{center}
In general, if $A_1, \ldots A_r$ are disjoint antidiagonals then the then the algorithm looks at each $A_i$ separately as they are part of separate components and the result is that $g_{A_1, \ldots A_r}=\det(A_1)\cdots\det(A_r)$.

If $A_1, \ldots A_r$ overlap to form one antidiagonal $X$ then the last step of the algorithm will occur only once and will produce $g_{A_1, \ldots A_r}=\det(X)$.  For example, 
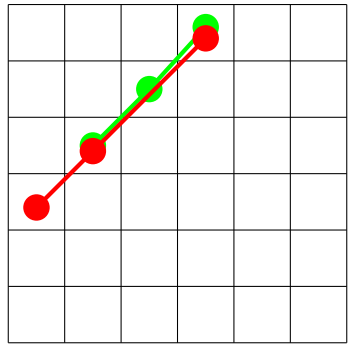
\begin{center}
\begin{tikzpicture}[x=.75cm, y=.75cm]
\draw (0,0)--(6,0);
\draw (0,1)--(6,1);
\draw (0,2)--(6,2);
\draw (0,3)--(6,3);
\draw (0,4)--(6,4);
\draw (0,5)--(6,5);
\draw (0,6)--(6,6);
\draw (0,0)--(0,6);
\draw (1,0)--(1,6);
\draw (2,0)--(2,6);
\draw (3,0)--(3,6);
\draw (4,0)--(4,6);
\draw (5,0)--(5,6);
\draw (6,0)--(6,6);
\draw [color=green, ultra thick] (1.5,3.5)--(2.5,4.5)--(3.5,5.6);
\fill [color=green] (1.5,3.5) circle(5pt);
\fill [color=green] (2.5,4.5) circle(5pt);
\fill [color=green] (3.5,5.6) circle(5pt);
\draw [color=red, ultra thick] (.5,2.4)--(1.5,3.4)--(3.5,5.4);
\fill [color=red] (.5,2.4) circle(5pt);
\fill [color=red] (1.5,3.4) circle(5pt);
\fill [color=red] (3.5,5.4) circle(5pt);
\node [right] at (7,3) {$\left|\begin{array}{cccc}m_{1,1} & m_{1,2} & m_{1,3} & m_{1,4} \\ m_{2,1} & m_{2,2} & m_{2,3} & m_{2,4} \\m_{3,1} & m_{3,2} & m_{3,3} & m_{3,4} \\ m_{4,1} & m_{4,2} & m_{4,3} & m_{4,4} \\ \end{array}\right|.$};
\end{tikzpicture}
\end{center}

In this example, there are two longest possible antidiagonals, the three cells occupied by the green dots and the three cells occupied by the red dots.  The ones occupied by the green dots are more northwest, hence the generator for the three antidiagonals shown below is 
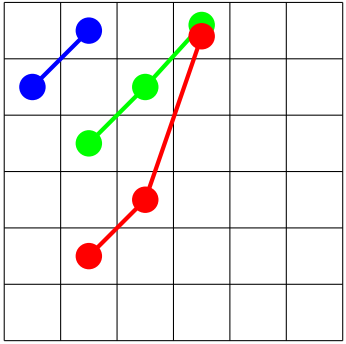
\begin{center}
\begin{tikzpicture}[x=.75cm, y=.75cm]
\draw (0,0)--(6,0);
\draw (0,1)--(6,1);
\draw (0,2)--(6,2);
\draw (0,3)--(6,3);
\draw (0,4)--(6,4);
\draw (0,5)--(6,5);
\draw (0,6)--(6,6);
\draw (0,0)--(0,6);
\draw (1,0)--(1,6);
\draw (2,0)--(2,6);
\draw (3,0)--(3,6);
\draw (4,0)--(4,6);
\draw (5,0)--(5,6);
\draw (6,0)--(6,6);
\draw [color=blue, ultra thick] (.5,4.5)--(1.5,5.5);
\fill [color=blue] (.5,4.5) circle(5pt);
\fill [color=blue] (1.5,5.5) circle(5pt);
\draw [color=green, ultra thick] (1.5,3.5)--(2.5,4.5)--(3.5,5.6);
\fill [color=green] (1.5,3.5) circle(5pt);
\fill [color=green] (2.5,4.5) circle(5pt);
\fill [color=green] (3.5,5.6) circle(5pt);
\draw [color=red, ultra thick] (1.5,1.5)--(2.5,2.5)--(3.5,5.4);
\fill [color=red] (1.5,1.5) circle(5pt);
\fill [color=red] (2.5,2.5) circle(5pt);
\fill [color=red] (3.5,5.4) circle(5pt);
\node [right] at (7,3) {$\left| \begin{array}{cc}m_{1,1} & m_{1,2} \\m_{2,1} & m_{2,2}\\ \end{array}\right| \left| \begin{array}{ccc}m_{1,2} & m_{1,3} & m_{1,3} \\m_{2,2} & m_{2,3} & m_{2,3} \\ m_{3,2} & m_{3,3} & m_{3,4} \\ \end{array}\right| \left| \begin{array}{cc}m_{4,2} & m_{4,3} \\m_{5,2} & m_{2,2}\end{array}\right| .$};
\end{tikzpicture}
\end{center}

In the picture below, the longest possible anti diagonal uses all of the cells in the green anti diagonal but only some of the cells in the red antidiagonal, however, there is only one possible longest antidiagonal.   Thus the generator is 
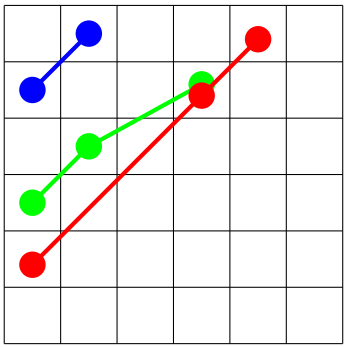
\begin{center}
\begin{tikzpicture}[x=.75cm, y=.75cm]
\draw (0,0)--(6,0);
\draw (0,1)--(6,1);
\draw (0,2)--(6,2);
\draw (0,3)--(6,3);
\draw (0,4)--(6,4);
\draw (0,5)--(6,5);
\draw (0,6)--(6,6);
\draw (0,0)--(0,6);
\draw (1,0)--(1,6);
\draw (2,0)--(2,6);
\draw (3,0)--(3,6);
\draw (4,0)--(4,6);
\draw (5,0)--(5,6);
\draw (6,0)--(6,6);
\draw [color=blue, ultra thick] (.5,4.5)--(1.5,5.5);
\fill [color=blue] (.5,4.5) circle(5pt);
\fill [color=blue] (1.5,5.5) circle(5pt);
\draw [color=green, ultra thick] (.5,2.5)--(1.5,3.5)--(3.5,4.6);
\fill [color=green] (.5,2.5) circle(5pt);
\fill [color=green] (1.5,3.5) circle(5pt);
\fill [color=green] (3.5,4.6) circle(5pt);
\draw [color=red, ultra thick] (.5,1.4)--(3.5,4.4)--(4.5,5.4);
\fill [color=red] (.5,1.4) circle(5pt);
\fill [color=red] (3.5,4.4) circle(5pt);
\fill [color=red] (4.5,5.4) circle(5pt);
\node [right] at (7,3){$\left|\begin{array}{cc}m_{1,1} & m_{1,2} \\m_{2,1} & m_{2,2}\\ \end{array}\right| \left|\begin{array}{cccc}m_{1,1} & m_{1,2} & m_{1,4} & m_{1,5} \\m_{2,1} & m_{2,2} & m_{2,4} & m_{2,5} \\m_{3,1} & m_{3,2} & m_{3,4} & m_{3,5} \\m_{4,1} & m_{4,2} & m_{4,4} & m_{4,5}\end{array}\right|\left|\begin{array}{c}m_{5,1}\end{array}\right|.$};
\end{tikzpicture}
\end{center}

We now give two examples where the complete ideals are comparatively small.  Firstly, we calculate $I(\overline{X}_{231} \cup \overline{X}_{312})=I(\overline{X}_{231})\cap I(\overline{X}_{312})$.  $I(\overline{X}_{231})=\langle m_{1,1}, m_{2,1} \rangle$ and $I(\overline{X}_{312})=\langle m_{1,1}, m_{1,2} \rangle$.  The antidiagonals and corresponding generators are shown below with antidiagonals from generators of $I(\overline{X}_{231})$ shown in red and antidiagonals of generators of  $I(\overline{X}_{312})$ shown in blue.  Note that the antidiagonals are only one cell each in this case.  
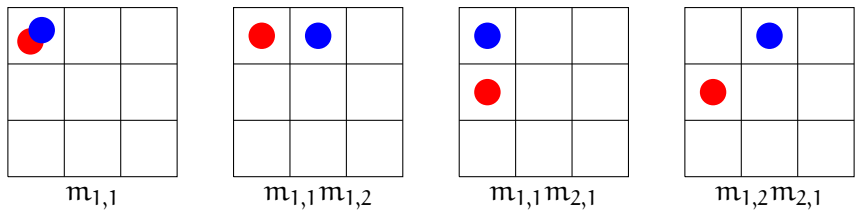
\begin{center}
\begin{tikzpicture}[x=.75cm, y=.75cm]
\draw (0,0)--(0,3)--(3,3)--(3,0)--(0,0);
\draw (0,1)--(3,1);
\draw (0,2)--(3,2);
\draw (1,0)--(1,3);
\draw (2,0)--(2,3);
\fill [color=red] (.4,2.4) circle(5pt);
\fill [color=blue] (.6,2.6) circle(5pt);
\node [below] at (1.5,0) {$m_{1,1}$};
\draw (4,0)--(4,3)--(7,3)--(7,0)--(4,0);
\draw (4,1)--(7,1);
\draw (4,2)--(7,2);
\draw (5,0)--(5,3);
\draw (6,0)--(6,3);
\fill [color=red] (4.5,2.5) circle(5pt);
\fill [color=blue] (5.5,2.5) circle(5pt);
\node [below] at (5.5,0) {$m_{1,1}m_{1,2}$};
\draw (8,0)--(8,3)--(11,3)--(11,0)--(8,0);
\draw (8,1)--(11,1);
\draw (8,2)--(11,2);
\draw (9,0)--(9,3);
\draw (10,0)--(10,3);
\fill [color=red] (8.5,1.5) circle(5pt);
\fill [color=blue] (8.5,2.5) circle(5pt);
\node [below] at (9.5,0) {$m_{1,1}m_{2,1}$};
\draw (12,0)--(12,3)--(15,3)--(15,0)--(12,0);
\draw (12,1)--(15,1);
\draw (12,2)--(15,2);
\draw (13,0)--(13,3);
\draw (14,0)--(14,3);
\fill [color=red] (12.5,1.5) circle(5pt);
\fill [color=blue] (13.5,2.5) circle(5pt);
\node [below] at (13.5,0) {$m_{1,2}m_{2,1}$};
\end{tikzpicture}
\end{center}
Theorem \ref{mainthm} results in 
\[
I(\overline{X}_{231} \cup \overline{X}_{312})=I(\overline{X}_{231})\cap I(\overline{X}_{312})=\langle m_{1,1}, m_{1,1}m_{1,2}, m_{1,1}m_{2,1}, m_{1,2}m_{2,1} \rangle.
\] 

As a slightly larger example, consider $I(\overline{X}_{1423} \cup \overline{X}_{1342})=I(\overline{X}_{1423}) \cap I(\overline{X}_{1342})$.  These generators are given below in the order that the antidiagonals are displayed reading left to right and top to bottom.  The antidiagonals for $I(\overline{X}_{1423}$ are shown in red while the antidigaonals $I(\overline{X}_{1342})$ are shown in blue.  for Note that the full $4 \times 4$ grid is not displayed, but only the northwest $3 \times 3$ portion where antidiagonals for these two ideals may lie.  
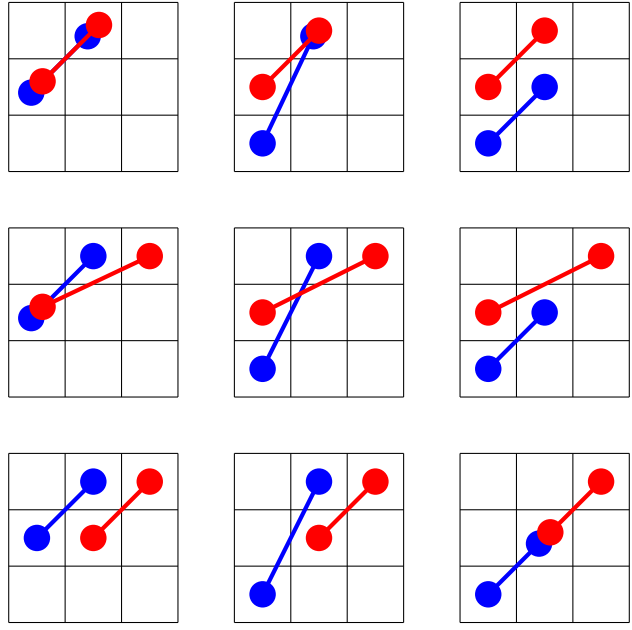
\begin{center}
\begin{tikzpicture}[x=.75cm, y=.75cm]
\draw (0,0)--(3,0);
\draw (4,0)--(7,0);
\draw (8,0)--(11,0);
\draw (0,1)--(3,1);
\draw (4,1)--(7,1);
\draw (8,1)--(11,1);
\draw (0,2)--(3,2);
\draw (4,2)--(7,2);
\draw (8,2)--(11,2);
\draw (0,3)--(3,3);
\draw (4,3)--(7,3);
\draw (8,3)--(11,3);
\draw (0,4)--(3,4);
\draw (4,4)--(7,4);
\draw (8,4)--(11,4);
\draw (0,5)--(3,5);
\draw (4,5)--(7,5);
\draw (8,5)--(11,5);
\draw (0,6)--(3,6);
\draw (4,6)--(7,6);
\draw (8,6)--(11,6);
\draw (0,7)--(3,7);
\draw (4,7)--(7,7);
\draw (8,7)--(11,7);
\draw (0,8)--(3,8);
\draw (4,8)--(7,8);
\draw (8,8)--(11,8);
\draw (0,9)--(3,9);
\draw (4,9)--(7,9);
\draw (8,9)--(11,9);
\draw (0,10)--(3,10);
\draw (4,10)--(7,10);
\draw (8,10)--(11,10);
\draw (0,11)--(3,11);
\draw (4,11)--(7,11);
\draw (8,11)--(11,11);
\draw (0,0)--(0,3);
\draw (0,4)--(0,7);
\draw (0,8)--(0,11);
\draw (1,0)--(1,3);
\draw (1,4)--(1,7);
\draw (1,8)--(1,11);
\draw (2,0)--(2,3);
\draw (2,4)--(2,7);
\draw (2,8)--(2,11);
\draw (3,0)--(3,3);
\draw (3,4)--(3,7);
\draw (3,8)--(3,11);
\draw (4,0)--(4,3);
\draw (4,4)--(4,7);
\draw (4,8)--(4,11);
\draw (5,0)--(5,3);
\draw (5,4)--(5,7);
\draw (5,8)--(5,11);
\draw (6,0)--(6,3);
\draw (6,4)--(6,7);
\draw (6,8)--(6,11);
\draw (7,0)--(7,3);
\draw (7,4)--(7,7);
\draw (7,8)--(7,11);
\draw (8,0)--(8,3);
\draw (8,4)--(8,7);
\draw (8,8)--(8,11);
\draw (9,0)--(9,3);
\draw (9,4)--(9,7);
\draw (9,8)--(9,11);
\draw (10,0)--(10,3);
\draw (10,4)--(10,7);
\draw (10,8)--(10,11);
\draw (11,0)--(11,3);
\draw (11,4)--(11,7);
\draw (11,8)--(11,11);
\fill [color=blue] (0.5,1.5) circle(5pt);
\fill [color=blue] (1.5,2.5) circle(5pt);
\draw [color=blue, ultra thick] (0.5,1.5)--(1.5,2.5);
\fill [color=blue] (0.4,5.4) circle(5pt);
\fill [color=blue] (1.5,6.5) circle(5pt);
\draw [color=blue, ultra thick] (0.4,5.4)--(1.5,6.5);
\fill [color=blue] (0.4,9.4) circle(5pt);
\fill [color=blue] (1.4,10.4) circle(5pt);
\draw [color=blue, ultra thick] (0.4,9.4)--(1.4,10.4);
\fill [color=blue] (4.5,.5) circle(5pt);
\fill [color=blue] (5.5,2.5) circle(5pt);
\draw [color=blue, ultra thick] (4.5,.5)--(5.5,2.5);
\fill [color=blue] (4.5,4.5) circle(5pt);
\fill [color=blue] (5.5,6.5) circle(5pt);
\draw [color=blue, ultra thick] (4.5,4.5)--(5.5,6.5);
\fill [color=blue] (4.5,8.5) circle(5pt);
\fill [color=blue] (5.4,10.4) circle(5pt);
\draw [color=blue, ultra thick] (4.5,8.5)--(5.4,10.4);
\fill [color=blue] (8.5,0.5) circle(5pt);
\fill [color=blue] (9.4,1.4) circle(5pt);
\draw [color=blue, ultra thick] (8.5,0.5)--(9.4,1.4);
\fill [color=blue] (8.5,4.5) circle(5pt);
\fill [color=blue] (9.5,5.5) circle(5pt);
\draw [color=blue, ultra thick] (8.5,4.5)--(9.5,5.5);
\fill [color=blue] (8.5,8.5) circle(5pt);
\fill [color=blue] (9.5,9.5) circle(5pt);
\draw [color=blue, ultra thick] (8.5,8.5)--(9.5,9.5);
\fill [color=red] (1.5,1.5) circle(5pt);
\fill [color=red] (2.5,2.5) circle(5pt);
\draw [color=red, ultra thick] (1.5,1.5)--(2.5,2.5);
\fill [color=red] (5.5,1.5) circle(5pt);
\fill [color=red] (6.5,2.5) circle(5pt);
\draw [color=red, ultra thick] (5.5,1.5)--(6.5,2.5);
\fill [color=red] (9.6,1.6) circle(5pt);
\fill [color=red] (10.5,2.5) circle(5pt);
\draw [color=red, ultra thick] (9.6,1.6)--(10.5,2.5);
\fill [color=red] (0.6,5.6) circle(5pt);
\fill [color=red] (2.5,6.5) circle(5pt);
\draw [color=red, ultra thick] (0.6,5.6)--(2.5,6.5);
\fill [color=red] (4.5,5.5) circle(5pt);
\fill [color=red] (6.5,6.5) circle(5pt);
\draw [color=red, ultra thick] (4.5,5.5)--(6.5,6.5);
\fill [color=red] (8.5,5.5) circle(5pt);
\fill [color=red] (10.5,6.5) circle(5pt);
\draw [color=red, ultra thick] (8.5,5.5)--(10.5,6.5);
\fill [color=red] (0.6,9.6) circle(5pt);
\fill [color=red] (1.6,10.6) circle(5pt);
\draw [color=red, ultra thick] (0.6,9.6)--(1.6,10.6);
\fill [color=red] (4.5,9.5) circle(5pt);
\fill [color=red] (5.5,10.5) circle(5pt);
\draw [color=red, ultra thick] (4.5,9.5)--(5.5,10.5);
\fill [color=red] (8.5,9.5) circle(5pt);
\fill [color=red] (9.5,10.5) circle(5pt);
\draw [color=red, ultra thick] (8.5,9.5)--(9.5,10.5);
\end{tikzpicture}
\end{center}
Here Theorem \ref{mainthm} produces
\begin{align*}
\phantom{+}&\bigg\langle  \left|\begin{array}{cc}m_{1,1} & m_{1,2} \\m_{2,1} & m_{2,2}\end{array}\right|, \left|\begin{array}{cc}m_{1,1} & m_{1,2} \\m_{2,1} & m_{2,2}\end{array}\right| \left|\begin{array}{c}m_{1,1}\end{array}\right|, \left|\begin{array}{cc}m_{3,1} & m_{1,2} \\m_{2,1} & m_{2,2}\end{array}\right|\left|\begin{array}{cc}m_{2,1} & m_{2,2} \\m_{3,1} & m_{2,2}\end{array}\right|\bigg\rangle\\
+&\bigg\langle \left|\begin{array}{cc}m_{1,1} & m_{1,2} \\m_{3,1} & m_{2,2}\end{array}\right|\left|\begin{array}{c}m_{1,1}\end{array}\right|,  \left|\begin{array}{cc}m_{1,1} & m_{1,3} \\m_{2,1} & m_{2,3}\end{array}\right|\left|\begin{array}{cc}m_{1,1} & m_{1,3} \\m_{3,1} & m_{3,2}\end{array}\right|, \left|\begin{array}{cc}m_{1,1} & m_{1,3} \\m_{3,1} & m_{2,3}\end{array}\right|\left|\begin{array}{cc}m_{2,1} & m_{2,2} \\m_{3,1} & m_{2,3}\end{array}\right|\bigg\rangle\\
+&\bigg\langle \left|\begin{array}{cc}m_{1,1} & m_{1,2} \\m_{2,1} & m_{3,2}\end{array}\right|\left|\begin{array}{cc}m_{1,2} & m_{1,3} \\m_{2,2} & m_{2,2}\end{array}\right|,  \left|\begin{array}{cc}m_{1,1} & m_{1,2} \\m_{3,1} & m_{2,3}\end{array}\right|\left|\begin{array}{cc}m_{1,2} & m_{1,3} \\m_{2,2} & m_{2,2}\end{array}\right|, \left|\begin{array}{ccc}m_{1,1} & m_{1,2} & m_{1,3} \\m_{2,1} & m_{2,2} & m_{2,3} \\m_{3,1} & m_{3,2} & m_{3,4}\end{array}\right|\bigg\rangle
\end{align*}

\section{Proof of Theorem \ref{mainthm}}\label{S:proof}

We now prove the main result of this paper, Theorem \ref{mainthm}, which states that the $g_{A_1, \ldots ,A_r}$ generate $I_1 \cap \cdots \cap I_r$.  

We begin with a few fairly general statements:
\begin{theorem}[\cite{KnutsonFrob}]\label{thm:initI}
If $\{I_i: i \in S\}$ are ideals generated by northwest rank conditions then $\init (\cap_{i \in S} I_i) = \cap_{i \in S} (\init I_i).$
\end{theorem}

\begin{lemma}[\cite{KnutsonMiller2005}]\label{lem:initcont}
If $J \subseteq K$ are homogeneous ideals in a polynomial ring such that $\init J = \init K$ then $J=K$.  
\end{lemma}

\begin{lemma}\label{lem:antdiagsGlue}
Let $I_A$ and $I_B$ be ideals that define schemes of northwest rank conditions and let $\det(A) \in I_A$ and $\det(B) \in I_B$ be determinants with antidiagonals $A$ and $B$ respectively such that $A \cup B=X$ and $A\cap B \ne \varnothing$.  Then $\det(X)$ is in $I_A \cap I_B$.  
\end{lemma}

\begin{proof}
Let $V_X=V(\det(X))$, $V_A=V(I_A)$ and $V_B=V(I_B)$ be the varieties corresponding to the ideals $\langle \det(X) \rangle$, $I_A$ and $I_B$.  It is enough to show that $V_A \subseteq V_X$ and $V_B \subseteq V_X$.  

We will show that given a matrix with antidiagonal $X$ with a sub-matrix with antidiagonal $A \subseteq X$ where the sub-matrix northwest of the cells occupied by $A$ has rank at most $\text{length}(A)-1$ then the full matrix has rank at most $\text{length}(X)-1$.  The corresponding statement for sub-matrix with antidiagonal $B$ can be proven by replacing $A$ with $B$ everywhere.  

The basic idea of this proof is that we know the rank conditions on the rows and columns northwest of those occupied by $A$.  The rank conditions given by $A$ then imply other rank conditions as adding either a row or a column to a sub-matrix can increase its rank by at most one.  

Let $k$ be the number of rows, also the number of columns in the antidiagonal $X$.  Let the length of $A$ be $l+1$, so the rank condition on all rows and columns northwest of those occupied by $A$ is at most $l$.  Assume that the rightmost column of $A$ is $c$ and the leftmost column of $A$ is $t+1$.  Notice that this implies that the bottom row occupied by $A$ is $k-t$, as the antidiagonal element in column $t+1$ is in row $k-t$.  Thus, the northwest $(k-t) \times c$ of matrices in $V_A$ has rank at most $l$.  

Notice $c \ge (t+1)+(l+1)$, with equality if $A$ occupies a continuous set of columns, so matrices in $V_A$ have rank at most $l$ in the northwest $(k-t) \times (t+l+2)$.  Adding $k-c \le k-(n-t-l-2)$ columns to this sub-matrix gives a $(k-t)\times k$ sub-matrix with rank at most $l+k-c \le r+(k-t-l-2)=k-t-2$.  Further, by the same principle, moving down $t$ rows, the northwest $k \times k$, i.e. the whole matrix with antidiagonal $X$, has rank at most $k-t-2+t=k-2$, hence has rank at most $k-1$ and so is in $V_X$.  

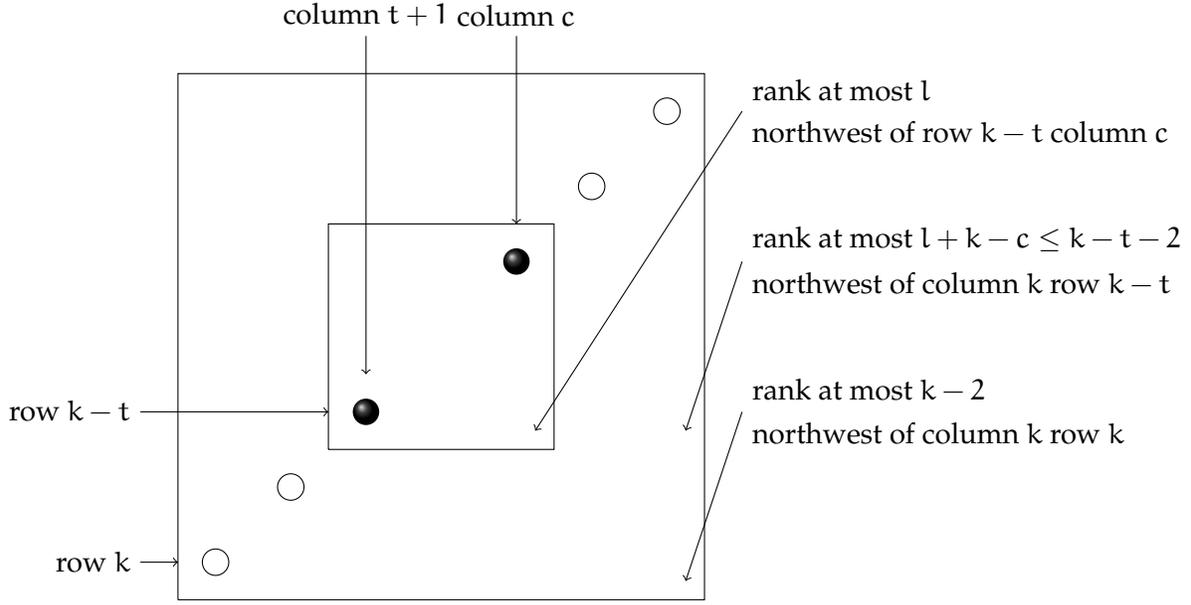
\begin{figure}
\begin{center}
\begin{tikzpicture}
\draw (0,0)--(0,7)--(7,7)--(7,0)--(0,0);
\draw (2,2)--(2,5)--(5,5)--(5,2)--(2,2);
\draw (.5,.5) circle(5pt) node{};
\draw (1.5,1.5) circle(5pt) node{};
\draw (5.5,5.5) circle(5pt) node{};
\draw (6.5,6.5) circle(5pt) node{};
\shade [ball color=black] (2.5,2.5) circle(5pt) node{};
\shade [ball color=black] (4.5,4.5) circle(5pt) node{};
\draw [->] (-.5, 2.5)--(2,2.5);
\node [left] at (-.5,2.5) {row $k-t$};
\draw [->] (-.5, .5)--(0,.5);
\node [left] at (-.5,.5) {row $k$};
\draw [->] (2.5,7.5)--(2.5,3);
\node [above] at (2.5,7.5) {column $t+1$};
\draw [->] (4.5,7.5)--(4.5,5);
\node [above] at (4.5,7.5) {column $c$};
\draw[->](7.5,6.5)--(4.75,2.25);
\node [above right] at (7.5,6.5) {rank at most $l$};
\node [below right] at (7.5,6.5) {northwest of row $k-t$ column $c$};
\draw[->](7.5,4.5)--(6.75,2.25);
\node [above right] at (7.5,4.5) {rank at most $l+k-c \le k-t-2$};
\node [below right] at (7.5,4.5) {northwest of column $k$ row $k-t$};
\draw[->](7.5,2.5)--(6.75,.25);
\node [above right] at (7.5,2.5) {rank at most $k-2$};
\node [below right] at (7.5,2.5) {northwest of column $k$ row $k$};
%\draw[fill] (4.5, 4.5) circle [radius=0.1]
\end{tikzpicture}
\caption{The proof of Lemma \ref{lem:antdiagsGlue}. The antidiagonal cells in $A$ are marked in black and the antidiagonal cells in $X- A \subseteq B$ are marked in white.  }
\label{fig:nwranklemma}
\end{center}
\end{figure}

\end{proof}

For a visual explanation of the proof of Lemma \ref{lem:antdiagsGlue} see Figure \ref{fig:nwranklemma}. 

%I'm not sure I need this next lemma, but I'll store it here for reference in case I do.  

%\begin{lemma}
%Let $A$ be (weakly) northwest of $B$ and occupying a subset of the rows of $B$.  If an $A$ intersects an antidiagonal $B$ at the most northeast dot of $A$ then the determinant with antidiagonal $B$ is in any ideal of northwest rank conditions that contains a determinant with antidiagonal $A$.  
%\end{lemma}

%\begin{proof}
%The rank condition that $A$ must have come from forces a rank condition of one less than the length of $A$ in rightmost column and lowest row that $A$ occupies.  This requires that a determinant with an antidiagonal of the same length of $A$ and coinciding with $B$ must also be in the ideal $A$ came from.  Now we can appeal to the previous lemma.  
%\end{proof}

\begin{lemma}\label{lem:ginI}
$g_{A_1, \ldots , A_r} \in I_{i}$ for $1 \le i \le r$ and hence 
\[
\langle g_{A_1, \dots, A_r}: A_i \text{ ranges over all antidiagonals for Fulton generators of } I_i \rangle \subseteq \cap_{i =1}^r I_i.
\]
\end{lemma}

\begin{proof}
Fix $i$.  Let $S$ be the first antidiagonal containing a box occupied by a box contained in $A_i$ added to $g_{A_1, \ldots , A_r}$.  We shall show that $\det(S)$ is in $I_i$ and hence $g_{A_1, \ldots , A_r} \in I_i$ as it is a multiple of $\det(S)$.  If $A_i \subseteq S$ then $\det(S) \in I_i$ either because $S=A_i$ or $S \subsetneq A_i$ in which case we apply Lemma  \ref{lem:antdiagsGlue}.  Otherwise, $|S| \ge |A_i|$ and $S$ is weakly to the northwest of $A_i$.   Therefore, there is a subset $B$ of $S$ such that $|B|=|A_i|$, and $B$ is weakly northwest of $A_i$.  Hence, $B$ is an antidiagonal for some determinant in $I_i$, and again by Lemma \ref{lem:antdiagsGlue} $\det(S) \in I_i$.  
\end{proof}

\begin{lemma}\label{initglemma}
$\init g_{A_1, \dots, A_r} =A_1 \cup \cdots \cup A_r$ under any antidiagonal term order.  
\end{lemma}

\begin{proof}
$\init g_{A_1, \dots, A_r}$ is a product of determinants, with collective antidiagonals $A_1 \cup \cdots \cup A_r$.  
\end{proof}

When we combine Lemma \ref{initglemma} and Theorem \ref{thm:initI} we see that $\init \langle g_{A_1, \dots A_r}\rangle = \init (\cap I_i$).  Then, Lemmas \ref{lem:initcont} and \ref{lem:ginI} combine to complete the proof of Theorem \ref{mainthm}.  

Note that Theorem \ref{mainthm} may produce an oversupply of generators.  For example, if $I_1=I_2$, then inputting the same set of $p$ Fulton generators twice results in a Gr\"obner basis of $p^2$ polynomials for $I_1 \cap I_2=I_1=I_2$.  

\bibliography{MatrixSchubertUnionsArxiv1Refs}
\bibliographystyle{amsalpha}

\end{document}